\documentclass[a4paper]{amsart}
\author[Raghavan]{Dilip Raghavan}
\thanks{This paper was completed when the first author was a Fields Research Fellow.
The first author thanks the Fields Institute for its kind hospitality.}
\address[Raghavan]{Department of Mathematics \\
National University of Singapore\\
Singapore 119076.}
\email{\href{dilip.raghavan@protonmail.com}{dilip.raghavan@protonmail.com}}
\urladdr{\url{https://dilip-raghavan.github.io/}}
\author[Todorcevic]{Stevo Todorcevic}
\thanks{Second author is partially supported by grants from NSERC (455916) and CNRS (IMJ-PRG UMR7586).}
\address[Todorcevic]{Department of Mathematics, University of Toronto, Toronto, Canada, M5S 2E4.}
\email{\href{stevo@math.toronto.edu}{stevo@math.toronto.edu}}
\address[Todorcevic]{Institut de Math\'{e}matique de Jussieu, UMR 7586, Case 247, 4 place Jussieu, 75252 Paris Cedex, France.}
\email{\href{todorcevic@math.jussieu.fr}{todorcevic@math.jussieu.fr}}
\address[Todorcevic]{Matemati\v{c}ki Institut, SANU, Belgrade, Serbia.}
\email{\href{stevo.todorcevic@sanu.ac.rs}{stevo.todorcevic@sanu.ac.rs}}
\date{\today}
\title[A property of rho-functions]{A combinatorial property of rho-functions}
\usepackage[only, llbracket, rrbracket]{stmaryrd}
\usepackage{amssymb, amsmath, amsthm, mathrsfs, enumitem, amsfonts, latexsym, bbm}
\usepackage{hyperref}
\def\polhk#1{\setbox0=\hbox{#1}{\ooalign{\hidewidth
    \lower1.5ex\hbox{`}\hidewidth\crcr\unhbox0}}}
\newtheorem{Theorem}{Theorem}
\newtheorem{Claim}[Theorem]{Claim}
\newtheorem{Lemma}[Theorem]{Lemma}
\newtheorem{Cor}[Theorem]{Corollary}

\theoremstyle{definition}
\newtheorem{Def}[Theorem]{Definition}

\theoremstyle{remark}

\newcommand{\restrict}{\mathord\upharpoonright}

\renewcommand{\[}{\left[}
\renewcommand{\]}{\right]}

\newcommand{\PPP}{\mathcal{P}}
\newcommand{\QQ}{\mathbb{Q}}
\newcommand{\lc}{\left|}
\newcommand{\rc}{\right|}

\newcommand\FIN{\mathrm{FIN}}

\DeclareMathOperator{\otp}{otp}

\DeclareMathOperator{\dom}{dom}

\newcommand{\Pset}{\mathcal{P}}

\newcommand{\UUU}{{\mathcal{U}}}

\newcommand{\III}{{\mathcal{I}}}

\newcommand{\TTT}{{\mathcal{T}}}

\newcommand{\RR}{\mathbb{R}}

\newcommand{\pr}[2]{\left\langle #1, #2 \right\rangle}
\newcommand{\seq}[4]{\left\langle {#1}_{#2}: #2 #3 #4 \right\rangle}

\newcommand{\rr}{\mathtt{r}}
\newcommand{\ball}[2]{{B}_{d}\left( {x}_{#1}, \frac{1}{#2+1} \right)}
\newcommand{\pc}[2]{{\[#1\]}^{#2}}
\begin{document}
\begin{abstract}
 We show that if $\TTT$ is any Hausdorff topology on ${\omega}_{1}$, then any subset of ${\omega}_{1}$ which is homeomorphic to the rationals under $\TTT$ can be refined to a homeomorphic copy of the rationals on which $\bar{\rho}$ is shift-increasing.
\end{abstract}
\maketitle
\section{Introduction} \label{sec:intro}
Todorcevic~\cite{squarebracket} introduced walks on ordinals and analyzed their characteristics through various functions, which are collectively known as \emph{rho-functions}.
The study of the properties of these rho-functions has been critical to constructing and understanding combinatorial structures on uncountable cardinals, especially the first uncountable cardinal ${\omega}_{1}$.
The monograph \cite{walks} presents numerous applications of rho-functions to diverse areas of mathematics.

An important and useful class of rho-functions are those that satisfy certain ultrametric triangle inequalities.
In \cite{walks}, Todorcevic showed the existence of such a function $\rho: \pc{{\kappa}^{+}}{2} \rightarrow \kappa$ for every regular $\kappa$.
Chapter 3 of \cite{walks} develops a detailed theory of such rho-functions in the case of the first uncountable cardinal -- i.e.\@ ${\kappa}^{+} = {\omega}_{1}$ -- and presents several applications, including the construction of gaps in $\Pset(\omega) \slash \FIN$.
We recall the following definitions.
\begin{Def} \label{def:rhobar}
 A sequence $\bar{C} = \seq{C}{\alpha}{<}{{\omega}_{1}}$ is called a \emph{$C$-sequence} if the following hold:
 \begin{enumerate}
  \item
  ${C}_{\alpha} \subseteq \alpha$;
  \item
  ${C}_{\alpha+1} = \{\alpha\}$;
  \item
  if $\alpha$ is a limit ordinal, then $\otp({C}_{\alpha}) = \omega$ and $\sup({C}_{\alpha}) = \alpha$.
 \end{enumerate}
 Given a fixed $C$-sequence $\bar{C}$, $\rho: {\[ {\omega}_{1} \]}^{2} \rightarrow \omega$ is defined by recursion as follows:
 \begin{align*}
  \rho(\alpha, \beta) = \max\left\{ \lc {C}_{\beta} \cap \alpha \rc, \rho\left( \alpha, \min\left( {C}_{\beta} \setminus \alpha \right)\right), \rho(\xi, \alpha): \xi \in {C}_{\beta} \cap \alpha \right\},
 \end{align*}
 for $\alpha < \beta < {\omega}_{1}$ with the boundary condition that $\rho(\alpha, \alpha) = 0$.
 $\bar{\rho}: {\[{\omega}_{1}\]}^{2} \rightarrow \omega$ is defined by
 \begin{align*}
  \bar{\rho}(\alpha, \beta) = {2}^{\rho( \alpha, \beta )} \cdot \left( 2 \cdot \lc \left\{ \xi \leq \alpha: \rho(\xi, \alpha) \leq \rho(\alpha, \beta) \right\} \rc + 1 \right).
 \end{align*}
\end{Def}
The following was proved in Lemma 3.2.2 of~\cite{walks}.
\begin{Lemma} \label{lem:rhobar}
 For any $\alpha < \beta < \gamma < {\omega}_{1}$,
 \begin{enumerate}
  \item
  $\bar{\rho}(\alpha, \gamma) \neq \bar{\rho}(\beta, \gamma)$;
  \item
  $\bar{\rho}(\alpha, \gamma) \leq \max\left\{ \bar{\rho}(\alpha, \beta), \bar{\rho}(\beta, \gamma) \right\}$;
  \item
  $\bar{\rho}(\alpha, \beta) \leq \max\left\{ \bar{\rho}(\alpha, \gamma), \bar{\rho}(\beta, \gamma) \right\}$.
 \end{enumerate}
\end{Lemma}
Lopez-Abad and Todorcevic~\cite{lopez-abad-todorcevic-13} introduced a sequence of higher-dimensional functions having analogous properties.
For each $i \leq n < \omega$, they defined a function ${f}^{(n)}_{i}: \pc{{\omega}_{n}}{i+1} \rightarrow {\omega}_{n-i}$, and used these functions to construct normalized weakly-null sequences of length ${\omega}_{n}$ without any unconditional subsequences.
Key to their construction was the fact that the ${f}^{(n)}_{i}$ could be made shift-increasing on some infinite subset of every infinite set.
\begin{Def} \label{def:shift-increasing-general}
 Suppose $n \geq 1$ is a natural number.
 Suppose $\gamma$ and $\delta$ are ordinals.
 For a function $f: \pc{\delta}{n} \rightarrow \gamma$, a subset $B \subseteq \delta$ is said to be \emph{$f$-shift-increasing} if for any ${\alpha}_{1} < \dotsb < {\alpha}_{n} < {\alpha}_{n+1}$ all belonging to $B$, $f(\{{\alpha}_{1}, \dotsc, {\alpha}_{n}\}) \leq f(\{{\alpha}_{2}, \dotsc, {\alpha}_{n+1}\})$.
\end{Def}
Lopez-Abad and Todorcevic showed in \cite{lopez-abad-todorcevic-13} that for every $i \leq n$ and every $A \in \pc{{\omega}_{n}}{{\aleph}_{0}}$, there exists $B \in \pc{A}{{\aleph}_{0}}$ such that $B$ is ${f}^{(n)}_{i}$-shift-increasing.
In particular, for every $A \in \pc{{\omega}_{1}}{{\aleph}_{0}}$, there exists $B \in \pc{A}{{\aleph}_{0}}$ such that $B$ is shift-increasing for ${f}^{(1)}_{1} = \bar{\rho}$.
In this paper we generalize this result to topologically large sets.
We are interested in the situation where $\TTT$ is a Hausdorff topology on ${\omega}_{1}$.
The main result of this paper shows that if $A \in \pc{{\omega}_{1}}{{\aleph}_{0}}$ is a homeomorphic copy of $\QQ$ under $\TTT$, then there exists $B \in \pc{A}{{\aleph}_{0}}$ such that $B$ is homeomorphic to $\QQ$ and $B$ is $\bar{\rho}$-shift-increasing.
An important difference between our situation and the one in \cite{lopez-abad-todorcevic-13} is that infinite sets of ordinals satisfy Ramsey's theorem for pairs, but as Baumgartner~\cite{baumtop} showed, the topological space $\QQ$ badly fails Ramsey's theorem.
For this reason, the proof of Theorem \ref{thm:shiftinc} below is considerably trickier than the corresponding result in \cite{lopez-abad-todorcevic-13}, which relies on Ramsey's theorem for infinite sets.
We expect our result will have further applications to topology and functional analysis.
\section{Notation} \label{sec:notation}
Our set-theoretic notation is standard.
For any $A$, $\Pset(A)$ denotes the powerset of $A$.
When $\kappa$ is a cardinal, $\pc{X}{\kappa}$ is $\{A \subseteq X: \lc A \rc = \kappa\}$, and $\pc{X}{< \kappa}$ denotes $\{A \subseteq X: \lc A \rc < \kappa\}$.

Given a set $a$, $\III$ is said to be an \emph{ideal} on $a$ if $\III$ is a subset of $\Pset(a)$ such that the following conditions hold: if $b \subseteq a$ is finite, then $b \in \III$; if $b \in \III$ and $c \subseteq b$, then $c \in \III$; if $b \in \III$ and $c \in \III$, then $b \cup c \in \III$; and $a \notin \III$.
The first condition is sometimes expressed by saying that $\III$ is \emph{non-principal}, and the last condition by saying that $\III$ is \emph{proper.}

For sets $A$ and $B$, ${A}^{B}$ is the collection of all functions from $B$ to $A$.
If $\delta$ is an ordinal, then ${A}^{< \delta} = {\bigcup}_{\gamma < \delta}{{A}^{\gamma}}$.
If $f$ is a function, then $\dom(f)$ is the domain of $f$, and if $X \subseteq \dom(f)$, then $f''X$ is the image of $X$ under $f$ -- that is, $f''X = \{f(x): x \in X\}$.

For $\sigma \in {\omega}^{< \omega}$ and $n \in \omega$, ${\sigma}^{\frown}{\langle n \rangle}$ is the concatenation of $\sigma$ with the one element sequence $\langle n \rangle$.
Formally, ${\sigma}^{\frown}{\langle n \rangle} = \sigma \cup \{ \pr{\dom(\sigma)}{n} \}$.
$T \subseteq {\omega}^{< \omega}$ is a \emph{subtree} if it is closed under initial segments, that is if $\forall \sigma \in T \forall k \leq \dom(\sigma)\[\sigma \restrict k \in T\]$.

If $d$ is a metric on $X$, then ${B}_{d}(y, \varepsilon)$ denotes $\{z \in X: d(y, z) < \varepsilon\}$, for all $y \in X$ and $\varepsilon \in \RR$.
A topological space $\pr{X}{\TTT}$ is \emph{dense-in-itself} if for each $x \in X$ and each open neighborhood $U$ of $x$, there exists $y \in U$ with $y \neq x$.
A theorem of Sierpi{\' n}ski (see \cite{engelking}) says that $\pr{X}{\TTT}$ is homeomorphic to $\QQ$ with its usual topology if and only if it is non-empty, countable, metrizable, and dense-in-itself.
\section{Getting $\bar{\rho}$ to be shift increasing on a copy of $\QQ$} \label{sec:main}
Even though our main result is about functions on ${\omega}_{1}$, its proof reduces to an analysis of functions on countable sets satisfying certain properties.
We will begin with the proof of this countable Ramsey theoretic statement, which could be useful in other contexts.

Assume $\rr: {\[\omega\]}^{2} \rightarrow \omega$ is a function with the following three properties:
\begin{enumerate}
 \item
 $\forall k, l, m \in \omega\[k < l < m \implies \rr(k,m) \neq \rr(l,m)\]$;
 \item
 $\forall k, l, m \in \omega\[k < l < m \implies \rr(k, l) \leq \max\left\{ \rr(k,m), \rr(l,m)\right\}\]$;
 \item
 $\forall k, l, m \in \omega\[k < l < m \implies \rr(k, m) \leq \max\left\{ \rr(k,l), \rr(l,m) \right\}\]$.
\end{enumerate}
It is easy to see that these properties of $\rr$ imply that for any $k, l, m \in \omega$ with $k < l < m$, if $\rr(k,m) > \rr(l,m)$, then $\rr(k,l)=\rr(k,m)$, and that if $\rr(k,l) > \rr(l,m)$, then $\rr(k,m)=\rr(k,l)$.
\begin{Def} \label{def:sinc}
 $B \subseteq \omega$ is \emph{$\rr$-shift-increasing} if
 \begin{align*}
  \forall k, l, m \in B\[k < l < m \implies \rr(k,l) \leq \rr(l,m)\].
 \end{align*}
\end{Def}
Assume that $X$ is a topological space and that $\seq{x}{n}{\in}{\omega}$ is a sequence of distinct points of $X$ (i.e.\@ ${x}_{n} = {x}_{m}$ if and only if $n=m$) with the property that the subspace $\left\{ {x}_{n}: n \in \omega \right\}$ is homeomorphic to $\QQ$.
Fix a metric $d$ on $\left\{ {x}_{n}: n \in \omega \right\}$ that is compatible with the subspace topology.
Observe that for each $\varepsilon > 0$ and each $n \in \omega$, ${B}_{d}({x}_{n}, \varepsilon)$ is also homeomorphic to $\QQ$.
\begin{Def} \label{def:scattered}
 $A \subseteq \omega$ is said to be \emph{scattered} if there is no $B \subseteq A$ so that $\left\{ {x}_{n}: n \in B \right\}$ is homeomorphic to $\QQ$.
 
 It is clear that $\III = \{A \subseteq \omega: A \ \text{is scattered}\}$ is a proper non-principal ideal on $\omega$.
 Define ${\III}^{+} = \Pset(\omega) \setminus \III$.
\end{Def}
\begin{Def} \label{def:Aij}
 For $i, j \in \omega$, define ${A}_{i, j} = \left\{n \in \omega: {x}_{n} \in \ball{i}{j} \right\}$.
\end{Def}
\begin{Lemma} \label{lem:dense}
 The following hold:
 \begin{enumerate}
  \item
  for all $A \subseteq \omega$, $A \in {\III}^{+}$ if and only if there exists $B \subseteq A$ such that $B \neq \emptyset$ and $\forall i \in B \forall j \in \omega \exists n \in B \cap {A}_{i, j}\[n \neq i\]$;
  \item
  $\forall A \in {\III}^{+} \exists B \subseteq A\[B \in {\III}^{+} \ \text{and} \ \forall i \in B \forall j \in \omega \[B \cap {A}_{i, j} \in {\III}^{+}\]\]$.
 \end{enumerate}
\end{Lemma}
\begin{proof}
 For (1): fix $A \subseteq \omega$.
 By definition, $A \in {\III}^{+}$ if and only if
 \begin{align*}
  \exists B \subseteq A\[\{{x}_{n}: n \in B\} \ \text{is homeomorphic to} \ \QQ\].
 \end{align*}
 Consider any $B \subseteq A$.
 By a theorem of Sierpinski, $\{{x}_{n}: n \in B\}$ is homeomorphic to $\QQ$ if and only if $\{{x}_{n}: n \in B\}$ is countable, metrizable, non-empty, and dense-in-itself.
 Since $\{{x}_{n}: n \in \omega\}$ is metrizable and countable, it suffices to show that $\{{x}_{n}: n \in B\}$ is dense-in-itself if and only if $\forall i \in B \forall j \in \omega \exists n \in B \cap {A}_{i, j}\[n \neq i\]$.
 First assume that $\{{x}_{n}: n \in B\}$ is dense-in-itself.
 Fix $i \in B$ and $j \in \omega$.
 Then $\ball{i}{j}$ is an open set in $\{{x}_{n}: n \in \omega\}$, and so $\ball{i}{j} \cap \{ {x}_{n}: n \in B \}$ is an open neighborhood of ${x}_{i}$ in $\{ {x}_{n}: n \in B \}$.
 As $\{ {x}_{n}: n \in B \}$ is dense-in-itself, there exists $y \in \ball{i}{j} \cap \{ {x}_{n}: n \in B \}$ with $y \neq {x}_{i}$.
 Thus $y={x}_{m}$ for some $m \in B$ with $i \neq m$.
 By definition of ${A}_{i, j}$, $m \in {A}_{i, j}$.
 This proves one direction.
 For the converse, assume $\forall i \in B \forall j \in \omega \exists n \in B \cap {A}_{i, j}\[n \neq i\]$.
 Consider ${x}_{i}$ for some $i \in B$ and some open subset $V$ of $\{ {x}_{n}: n \in B \}$ with ${x}_{i} \in V$.
 Then $V = U \cap \{{x}_{n}: n \in B\}$ for some open set $U$ in $\{{x}_{n}: n \in \omega\}$.
 Thus $\ball{i}{j} \subseteq U$ for some $j \in \omega$.
 By the assumption there is $n \in B \cap {A}_{i, j}$ with $n \neq i$.
 As $n, i \in \omega$ and $n \neq i$, $y={x}_{n} \neq {x}_{i}$.
 By definition of ${A}_{i, j}$, $y={x}_{n} \in \ball{i}{j}$, $y={x}_{n} \in \{{x}_{m}: m \in B\}$, whence $y \in U \cap \{{x}_{m}: m \in B\} = V$.
 As $y \in V$ and $y \neq {x}_{i}$, this shows $\{{x}_{m}: m \in B\}$ is dense-in-itself, proving (1).
 
 For (2): fix $A \in {\III}^{+}$.
 Applying (1) to $A$, there exists $B \subseteq A$ so that $B \neq \emptyset$ and $\forall i \in B \forall j \in \omega \exists n \in B \cap {A}_{i, j}\[n \neq i\]$.
 Applying (1) to $B$, we see that $B \in {\III}^{+}$.
 Fix $i \in B$ and $j \in \omega$.
 To see $B \cap {A}_{i, j} \in {\III}^{+}$, we apply (1) again.
 We have $B \cap {A}_{i, j} \subseteq B \cap {A}_{i, j}$ and by the choice of $B$, $\exists n \in B \cap {A}_{i, j}\[n \neq i\]$, which implies $B \cap {A}_{i, j} \neq \emptyset$.
 Fix $k \in B \cap {A}_{i, j}$ and $l \in \omega$.
 It suffices to find $m \in B \cap {A}_{i, j} \cap {A}_{k, l}$ with $m \neq k$.
 By the definition of ${A}_{i, j}$, $d({x}_{k}, {x}_{i}) < \frac{1}{j+1}$.
 Choose $q \in \omega$ so that $\frac{1}{q+1} < \frac{1}{l+1}$ and $d({x}_{k}, {x}_{i}) + \frac{1}{q+1} < \frac{1}{j+1}$.
 By the choice of $B$, there exists $m \in B \cap {A}_{k, q}$ with $m \neq k$.
 Thus $m \in \omega$ and $d({x}_{m}, {x}_{k}) < \frac{1}{q+1} < \frac{1}{l+1}$, whence $m \in {A}_{k, l}$.
 Also $d({x}_{m}, {x}_{i}) \leq d({x}_{m}, {x}_{k}) + d({x}_{k}, {x}_{i}) < \frac{1}{q+1} + d({x}_{k}, {x}_{i}) < \frac{1}{j+1}$, whence $m \in {A}_{i, j}$.
 Therefore $m \in B \cap {A}_{i, j} \cap {A}_{k, l}$, as required.
 This concludes the proof that $B \cap {A}_{i, j} \in {\III}^{+}$.
\end{proof}
\begin{Theorem} \label{thm:shiftinc}
 For every $A \in {\III}^{+}$, there exists $B \subseteq A$ such that $B$ is $\rr$-shift-increasing, $B \neq \emptyset$, and $\forall i \in B \forall j \in \omega \exists n \in B \cap {A}_{i, j}\[n \neq i\]$.
\end{Theorem}
\begin{proof}
 We will ensure $B$ has the following property:
 \begin{align*}
  \forall k, l, m \in B \[k < l < m \implies \rr(k, m) < \rr(l, m)\].
\end{align*} 
 To see that this implies that $B$ is $\rr$-shift-increasing, assume for a contradiction that for some $k, l, m \in B$ with $k < l < m$, $\rr(k, l) > \rr(l, m)$.
 Then by the properties of $\rr$ discussed earlier, $\rr(k, m) = \rr(k, l) > \rr(l, m)$, contradicting the property of $B$.
 
 Fix a $1-1$ enumeration $\seq{\sigma}{s}{<}{\omega}$ of ${\omega}^{< \omega}$ such that $\forall s < s' < \omega\[{\sigma}_{s'} \not\subseteq {\sigma}_{s} \]$.
 Note that ${\sigma}_{0} = \emptyset$ and that for each $s > 0$, there exist unique $r < s$ and $j \in \omega$ so that ${\sigma}_{s} = {\sigma}^{\frown}_{r}{\langle j \rangle}$.
 Applying (2) of Lemma \ref{lem:dense}, fix $D \subseteq A$ so that $D \in {\III}^{+}$ and $\forall i \in D \forall j \in \omega \[D \cap {A}_{i, j} \in {\III}^{+}\]$.
 Construct $\seq{k}{s}{<}{\omega}$ and $\seq{\UUU}{s}{<}{\omega}$ with the following properties:
 \begin{enumerate}
  \item
  ${k}_{s} \in D$, $\forall r < s \[ {k}_{r} < {k}_{s} \]$, $\forall q, r \[q < r < s \implies \rr({k}_{q}, {k}_{s}) < \rr({k}_{r}, {k}_{s}) \]$;
  \item
  if $s > 0$ and ${\sigma}_{s} = {\sigma}^{\frown}_{r}{\langle j \rangle}$ for some $r < s$, then ${k}_{s} \in {A}_{{k}_{r}, j}$;
  \item
  ${\UUU}_{s} \subseteq {\III}^{+}$ is an ultrafilter on $\omega$ such that $\forall j \in \omega \[D \cap {A}_{{k}_{s}, j} \in {\UUU}_{s}\]$;
  \item
  $\forall p, q, r \leq s \[ p < q \implies \{m \in D: m > {k}_{q} \ \text{and} \ \rr({k}_{p}, m) < \rr({k}_{q}, m)\} \in {\UUU}_{r} \]$.
 \end{enumerate}
 Suppose for a moment that this construction can be carried out.
 Put $B = \{ {k}_{s}: s < \omega \}$.
 Then clearly $B$ is non-empty, and (1) ensures that $B \subseteq D \subseteq A$ and that $B$ satisfies the property claimed in the first paragraph of the proof.
 Consider $i \in B$ and $j \in \omega$.
 Then $i={k}_{r}$ for some $r < \omega$, and ${\sigma}^{\frown}_{r}{\langle j \rangle} = {\sigma}_{s}$ for some $r < s < \omega$.
 By (2) $n={k}_{s} \in B \cap {A}_{{k}_{r}, j} = B \cap {A}_{i, j}$, and by (1) $n = {k}_{s} > {k}_{r} = i$.
 Thus $B$ has the required properties.
 
 To construct $\seq{k}{s}{<}{\omega}$ and $\seq{\UUU}{s}{<}{\omega}$, proceed by induction.
 When $s=0$, let ${k}_{s} \in D$ be arbitrary.
 By the choice of $D$, $\forall j \in \omega \[ D \cap {A}_{{k}_{s}, j} \in {\III}^{+} \]$.
 Since $\forall j < j' < \omega \[{A}_{{k}_{s}, j'} \subseteq {A}_{{k}_{s}, j} \]$, $\left\{ D \cap {A}_{{k}_{s}, j}: j \in \omega \right\}$ forms a descending collection of elements of ${\III}^{+}$.
 Therefore, it is possible to find an ultrafilter ${\UUU}_{s}$ on $\omega$ such that $\left\{ D \cap {A}_{{k}_{s}, j}: j \in \omega \right\} \subseteq {\UUU}_{s} \subseteq {\III}^{+}$.
 This fulfils (1)--(4) for $s=0$.

 Now assume that $s \in \omega$ and that $\seq{k}{r}{\leq}{s}$ and $\seq{\UUU}{r}{\leq}{s}$ satisfying (1)--(4) for all $r \leq s$ are given.
 For some unique $r \leq s$ and $j \in \omega$, ${\sigma}_{s+1} = {\sigma}^{\frown}_{r}{\langle j \rangle}$.
 By (3) ${\UUU}_{r} \subseteq {\III}^{+}$ is an ultrafilter on $\omega$ with $D \cap {A}_{{k}_{r}, j} \in {\UUU}_{r}$.
 The following simple but useful claim is a corollary to Lemma \ref{lem:dense}.
 \begin{Claim} \label{claim:shift1}
  $\forall C \in {\UUU}_{r} \exists I \in \III \forall i \in C \setminus I \forall w \in \omega \[ \left( C \setminus I \right) \cap {A}_{i, w} \in {\III}^{+} \]$.
 \end{Claim}
 \begin{proof}
  Put $I = \left\{ i \in C: \exists w \in \omega \[ C \cap {A}_{i, w} \notin {\III}^{+} \] \right\}$.
  To see that $I \in \III$, suppose for a contradiction that $I \in {\III}^{+}$.
  Applying (2) of Lemma \ref{lem:dense}, find $J \subseteq I$ such that $J \in {\III}^{+}$ and $\forall i \in J \forall w \in \omega\[ J \cap {A}_{i, w} \in  {\III}^{+} \]$.
  As $J$ is non-empty, fix some $i \in J$.
  Then $i \in C \subseteq \omega$ and for some $w \in \omega$, $C \cap {A}_{i, w} \notin {\III}^{+}$.
  Also, $J \cap {A}_{i, w} \in {\III}^{+}$.
  However this is a contradiction because $J \cap {A}_{i, w} \subseteq I \cap {A}_{i, w} \subseteq C \cap {A}_{i, w} \subseteq C \subseteq \omega$.
  Thus $I \in \III$.
  To see that $I$ has the required properties, fix $i \in C \setminus I$ and $w \in \omega$.
  Then $C \cap {A}_{i, w} \in {\III}^{+}$ by definition of $I$.
  Therefore, $\left( C \setminus I \right) \cap {A}_{i, w} = \left( C \cap {A}_{i, w} \right) \setminus I \in {\III}^{+}$, as required.
 \end{proof}
 For each $p, q \leq s$ with $p < q$, define
 \begin{align*}
  {E}_{p, q} = \{ m \in D: m > {k}_{q} \ \text{and} \ \rr({k}_{p}, m) < \rr({k}_{q}, m) \} \in {\UUU}_{r}.
 \end{align*}
 Define $E = \bigcap \left( \left\{ D \cap {A}_{{k}_{r}, j} \right\} \cup \left\{ {E}_{p, q}: p, q \leq s \ \text{and} \ p < q \right\} \right) \in {\UUU}_{r}$.
 For each $p \leq s$ and $i \in E$ with ${k}_{p} < i$, define ${F}_{p, i} = \left\{ m \in D: i < m \ \text{and} \ \rr({k}_{p}, m) < \rr(i, m) \right\}$.
 Define
 \begin{align*}
  {G}_{p} = \left\{ i \in E: {k}_{p} < i \ \text{and} \ \forall w \in \omega \[ E \cap {F}_{p, i} \cap {A}_{i, w} \in {\III}^{+} \] \right\}.
 \end{align*}
 \begin{Claim} \label{claim:shift2}
  $\forall p \leq s \[ {G}_{p} \in {\UUU}_{r} \]$.
 \end{Claim}
 \begin{proof}
  Suppose not and fix a counterexample $p \leq s$.
  Since $\{ i \in E: i > {k}_{p} \} \in {\UUU}_{r}$, it follows that ${\bar{G}}_{p} = \left\{ i \in E: i > {k}_{p} \ \text{and} \ \exists w \in \omega \[ E \cap {F}_{p, i} \cap {A}_{i, w} \in \III \] \right\} \in {\UUU}_{r}$.
  Using Claim \ref{claim:shift1} fix ${I}_{0} \in \III$ so that $\forall i \in H \forall w \in \omega \[H \cap {A}_{i, w} \in {\III}^{+} \]$, where $H = {\bar{G}}_{p} \setminus {I}_{0}$.
  Since ${\bar{G}}_{p} \in {\UUU}_{r}$ and ${I}_{0} \in \III$, $H \in {\UUU}_{r}$.
  Consider some $i \in H$.
  Then $i \in E$, ${k}_{p} < i$, and for some $w \in \omega$, ${I}_{1} = E \cap {F}_{p, i} \cap {A}_{i, w} \in \III$, while $H \cap {A}_{i, w} \in {\III}^{+}$.
  As $\left( H \cap {A}_{i, w} \right) \setminus {I}_{1} \in {\III}^{+}$, we may select $i' \in \left( H \cap {A}_{i, w} \right) \setminus {I}_{1}$ with $i < i'$.
  Then $i' \in E \cap {A}_{i, w}$, whence $i' \notin {F}_{p, i}$.
  Since $E \subseteq D$, $i' \in D$ and $i < i'$, whence $\rr({k}_{p}, i') > \rr(i, i')$.
  As ${k}_{p} < i < i'$, the properties of $\rr$ imply that $\rr(i, i') < \rr({k}_{p}, i') = \rr({k}_{p}, i)$.
  We have thus proved that
  \begin{align*}
   \forall i \in H \exists i' \in H \[ i < i' \ \text{and} \ \rr(i, i') < \rr({k}_{p}, i') = \rr({k}_{p}, i) \].
  \end{align*}
 Now $H$ being non-empty, we may fix ${i}_{0} \in H$ and put $u = \rr({k}_{p}, {i}_{0})$.
 Construct ${i}_{0} < {i}_{1} < \dotsb < {i}_{u+1}$ so that for each $v \leq u+1$, ${i}_{v} \in H$ and $u = \rr({k}_{p}, {i}_{v})$ as follows.
 Suppose $v < u + 1$ and that ${i}_{v} \in H$ with $\rr({k}_{p}, {i}_{v}) = u$.
 Applying the property proved above we can find ${i}_{v+1} \in H$ with ${i}_{v} < {i}_{v+1}$ and $\rr({i}_{v}, {i}_{v+1}) < \rr({k}_{p}, {i}_{v+1}) = \rr({k}_{p}, {i}_{v}) = u$.
 By construction for each $v < u + 1$, $\rr({i}_{v}, {i}_{v+1}) < u$.
 By property (3) of $\rr$, this implies that $\forall v < u+1 \[ \rr({i}_{v}, {i}_{u+1}) < u \]$.
 As $\left\{ {i}_{0}, \dotsc, {i}_{u} \right\}$ is a set of size $u+1$, the pigeonhole principle implies that for some $0 \leq v < v' \leq u$, $\rr({i}_{v}, {i}_{u+1}) = \rr({i}_{v'}, {i}_{u+1})$, contradicting property (1) of $\rr$.
 This contradiction concludes the proof of the claim.
 \end{proof}
 For each $r', p \leq s$, define ${H}_{r', p} = \left\{ i \in E: i > {k}_{p} \ \text{and} \ {F}_{p, i} \in {\UUU}_{r'} \right\}$.
 \begin{Claim} \label{claim:shift3}
  For each $r', p \leq s$, ${H}_{r', p} \in {\UUU}_{r}$.
 \end{Claim}
 \begin{proof}
  Suppose not and fix some counterexample $r', p \leq s$.
  Since $\left\{ i \in E: i > {k}_{p} \right\} \in {\UUU}_{r}$, it follows that ${\bar{H}}_{r', p} = \left\{ i \in E: i > {k}_{p} \ \text{and} \ {F}_{p, i} \notin {\UUU}_{r'} \right\} \in {\UUU}_{r}$.
  Consider some $i \in E$ with $i > {k}_{p}$ and ${F}_{p, i} \notin {\UUU}_{r'}$.
  By (3) of the induction hypothesis applied to $r' \leq s$, $D \cap {A}_{{k}_{r'}, 0} \in {\UUU}_{r'}$, and so $\{ m \in D: m > i \} \in {\UUU}_{r'}$.
  Therefore, ${\bar{F}}_{p, i} = \left\{ m \in D: m > i \ \text{and} \ \rr({k}_{p}, m) > \rr(i, m) \right\} \in {\UUU}_{r'}$.
  Now ${\bar{H}}_{r', p}$ is an infinite subset of $\omega$.
  Fix any ${i}_{0} \in {\bar{H}}_{r', p}$ and put $a = \rr({k}_{p}, {i}_{0})$.
  Observe that for any $m \in {\bar{F}}_{p, {i}_{0}}$, ${k}_{p} < {i}_{0} < m$ and $\rr({i}_{0}, m) < \rr({k}_{p}, m)$, whence $\rr({k}_{p}, m) = \rr({k}_{p}, {i}_{0}) = a$.
  Choose $\left\{ {i}_{1}, \dotsc, {i}_{a+1} \right\} \subseteq {\bar{H}}_{r', p}$ such that ${i}_{0} < {i}_{1} < \dotsb < {i}_{a+1}$.
  Let $F = \bigcap \left\{ {\bar{F}}_{p, {i}_{b}}: 0 \leq b \leq a+1 \right\} \in {\UUU}_{r'}$, and fix $m \in F$.
  As $m \in {\bar{F}}_{p, {i}_{0}}$, $\rr({k}_{p}, m) = a$.
  For each $1 \leq b \leq a+1$, as $m \in {\bar{F}}_{p, {i}_{b}}$, ${i}_{b} < m$ and $\rr({i}_{b}, m) < \rr({k}_{p}, m) = a$.
  Hence by the pigeonhole principle, for some $1 \leq b < b' \leq a+1$, $\rr({i}_{b}, m) = \rr({i}_{b'}, m)$, contradicting property (1) of $\rr$.
  This contradiction proves the claim.
 \end{proof}
 Let
 \begin{align*}
  H = \left( \bigcap \left\{ {G}_{p}: p \leq s \right\} \right) \cap \left( \bigcap \left\{ {H}_{r', p}: r', p \leq s \right\} \right) \in {\UUU}_{r}.
 \end{align*}
 Choose ${k}_{s+1} \in H$.
 Then ${k}_{s+1} \in E$ and $\forall p \leq s \[ {k}_{p} < {k}_{s+1} \]$.
 Put
 \begin{align*}
  a = \max\left\{ \rr({k}_{p}, {k}_{s+1}): p \leq s \right\},
 \end{align*}
 and suppose $q \leq s$ is such that $a = \rr({k}_{q}, {k}_{s+1})$.
 As ${k}_{s+1} \in {G}_{q}$,
 \begin{align*}
  \forall w \in \omega \[ E \cap {F}_{q, {k}_{s+1}} \cap {A}_{{k}_{s+1}, w} \in {\III}^{+} \].
 \end{align*}
 Since $\forall w < w' < \omega \[ {A}_{{k}_{s+1}, w'} \subseteq {A}_{{k}_{s+1}, w} \]$, $\left\{ E \cap {F}_{q, {k}_{s+1}} \cap {A}_{{k}_{s+1}, w}: w \in \omega \right\}$ forms a descending sequence of members of ${\III}^{+}$.
 Therefore, there exists an ultrafilter ${\UUU}_{s+1}$ on $\omega$ such that $\left\{ E \cap {F}_{q, {k}_{s+1}} \cap {A}_{{k}_{s+1}, w}: w \in \omega \right\} \subseteq {\UUU}_{s+1} \subseteq {\III}^{+}$.
 We have $E \cap {F}_{q, {k}_{s+1}} \cap {A}_{{k}_{s+1}, 0} \subseteq E \subseteq \omega$ and $E \cap {F}_{q, {k}_{s+1}} \cap {A}_{{k}_{s+1}, 0} \subseteq {F}_{q, {k}_{s+1}} \subseteq D \subseteq \omega$, whence $E \in {\UUU}_{s+1}$ and ${F}_{q, {k}_{s+1}} \in {\UUU}_{s+1}$.
 Similarly for each $w \in \omega$, $E \cap {F}_{q, {k}_{s+1}} \cap {A}_{{k}_{s+1}, w} \subseteq D \cap {A}_{{k}_{s+1}, w} \subseteq D \subseteq \omega$, and so $\forall w \in \omega \[ D \cap {A}_{{k}_{s+1}, w} \in {\UUU}_{s+1} \]$.
 For every $p \leq s$, if $m \in {F}_{q, {k}_{s+1}}$, then $m \in D$, ${k}_{s+1} < m$, and $\rr({k}_{q}, m) < \rr({k}_{s+1}, m)$.
 Then ${k}_{p} < {k}_{s+1} < m$, and by properties (2) and (3) of $\rr$ and by the choice of $q$, $\rr({k}_{p}, m) \leq \max\left\{ \rr({k}_{p}, {k}_{s+1}), \rr({k}_{s+1}, m) \right\} \leq \max\left\{ \rr({k}_{q}, {k}_{s+1}), \rr({k}_{s+1}, m)  \right\} \leq \max \left\{ \max \left\{ \rr({k}_{q}, m), \rr({k}_{s+1}, m) \right\}, \rr({k}_{s+1}, m) \right\} = \rr({k}_{s+1}, m)$.
 By property (1) of $\rr$, we conclude that $\rr({k}_{p}, m) < \rr({k}_{s+1}, m)$, and hence that $m \in {F}_{p, {k}_{s+1}}$.
 Therefore, for every $p \leq s$, ${F}_{q, {k}_{s+1}} \subseteq {F}_{p, {k}_{s+1}} \subseteq D \subseteq \omega$.
 Therefore, $\forall p \leq s \[ {F}_{p, {k}_{s+1}} \in {\UUU}_{s+1} \]$.
 
 Unfix $q$ from the last paragraph.
 Let us verify that (1)--(4) are satisfied by $\seq{k}{r'}{\leq}{s+1}$ and $\seq{\UUU}{r'}{\leq}{s+1}$.
 We have noted above that ${k}_{s+1} \in E \subseteq D$ and that $\forall r' \leq s \[ {k}_{r'} < {k}_{s+1} \]$.
 Consider $p, q \leq s$ with $p < q$.
 Then ${E}_{p, q}$ is defined and ${k}_{s+1} \in {E}_{p, q}$, whence $\rr({k}_{p}, {k}_{s+1}) < \rr({k}_{q}, {k}_{s+1})$.
 This verifies (1).
 (2) holds because ${\sigma}_{s+1} = {\sigma}^{\frown}_{r}{\langle j \rangle}$, where $r \leq s$ and $j \in \omega$ are unique, and ${k}_{s+1} \in E \subseteq D \cap {A}_{{k}_{r}, j}$.
 For (3), by definition, ${\UUU}_{s+1}$ is an ultrafilter on $\omega$ with ${\UUU}_{s+1} \subseteq {\III}^{+}$, and we have noted above that $\forall w \in \omega \[ D \cap {A}_{{k}_{s+1}, w} \in {\UUU}_{s+1} \]$.
 Finally, we turn to (4).
 Fix $p, q, r' \leq s+1$ with $p<q$ and define ${L}_{p, q} = \left\{ m \in D: m > {k}_{q} \ \text{and} \ \rr({k}_{p}, m) < \rr({k}_{q}, m) \right\}$.
 We must show ${L}_{p, q} \in {\UUU}_{r'}$.
 Note that since $p \leq s$, there are four cases to consider.
 Suppose first that $q, r' \leq s$.
 Then by the induction hypothesis (4) applied to $s$, ${L}_{p, q} \in {\UUU}_{r'}$.
 Suppose next that $q \leq s$ and $r' = s+1$.
 Then ${E}_{p, q}$ is defined and ${E}_{p, q} = {L}_{p, q}$.
 Since $E \subseteq {E}_{p, q} \subseteq D \subseteq \omega$, and since $E \in {\UUU}_{s+1}$, ${E}_{p, q} \in {\UUU}_{s+1}$ as well.
 Thirdly, suppose $q=s+1$ and $r' \leq s$.
 Then ${H}_{r', p}$ is defined, and since ${k}_{s+1} \in {H}_{r', p}$, ${k}_{s+1} \in E$, ${k}_{s+1} > {k}_{p}$ and ${F}_{p, {k}_{s+1}} \in {\UUU}_{r'}$.
 Since by definition, ${L}_{p, q} = \left\{ m \in D: m > {k}_{s+1} \ \text{and} \ \rr({k}_{p}, m) < \rr({k}_{s+1}, m) \right\} = {F}_{p, {k}_{s+1}}$, ${L}_{p, q} \in {\UUU}_{r'}$ as well.
 Finally suppose that $q=s+1=r'$.
 Then since ${k}_{s+1} \in E$ and ${k}_{p} < {k}_{s+1}$, ${L}_{p, q} = \left\{ m \in D: m > {k}_{s+1} \ \text{and} \ \rr({k}_{p}, m) < \rr({k}_{s+1}, m) \right\} = {F}_{p, {k}_{s+1}}$, and since we have showed in the previous paragraph that ${F}_{p, {k}_{s+1}} \in {\UUU}_{s+1}$, ${L}_{p, q} \in {\UUU}_{s+1} = {\UUU}_{r'}$ as well.
 This concludes the verification of (1)--(4).
 Therefore the induction can proceed.
\end{proof}
\begin{Cor} \label{cor:famP}
 Let $\pr{X}{\TTT}$ be a topological space.
 Suppose that $\seq{x}{\alpha}{<}{{\omega}_{1}}$ is a 1--1 enumeration of all the points of $X$.
 Let $\PPP \subseteq {\[{\omega}_{1}\]}^{< {\aleph}_{1}}$ be a family such that:
 \begin{enumerate}
  \item
  $\PPP$ is \emph{hereditary}, that is, $\forall A \in \PPP \forall B \subseteq A \[B \in \PPP\]$;
  \item
  there exists $A \in \PPP$ such that the subspace $\{{x}_{\alpha}: \alpha \in A\}$ is homeomorphic to $\QQ$.
 \end{enumerate}
 Then there exists $A \in \PPP$ such that $\otp(A) = \omega$, $A$ is $\bar{\rho}$-shift-increasing, and the subspace $\{ {x}_{\alpha}: \alpha \in A \}$ is homeomorphic to $\QQ$.
\end{Cor}
\begin{proof}
 It is not hard to show that there is an $M \in \PPP$ such that $\otp(M) = \omega$ and the subspace $\{{x}_{\alpha}: \alpha \in M\}$ is homeomorphic to $\QQ$.
 Indeed this is proved in Lemma 6 of \cite{galvin-above-two}.
 Let $\{{\alpha}_{n}: n \in \omega\}$ be the strictly increasing enumeration of $M$.
 Define ${y}_{n} = {x}_{{\alpha}_{n}} \in X$, for all $n \in \omega$.
 Then $\seq{y}{n}{\in}{\omega}$ is a sequence of distinct points of $X$ and the subspace $\{{y}_{n}: n \in \omega\}$ is homeomorphic to $\QQ$.
 Let $d$ be a metric on $\{{y}_{n}: n \in \omega\}$ that is compatible with the subspace topology.
 Define $\rr: {\[\omega\]}^{2} \rightarrow \omega$ by setting $\rr(k, l) = \bar{\rho}({\alpha}_{k}, {\alpha}_{l})$, for all $k < l < \omega$.
 Let the ideal $\III$ and the sets ${A}_{i, j}$ be as in Definitions \ref{def:scattered} and \ref{def:Aij}.
 Then Theorem \ref{thm:shiftinc} applies and implies that there exists a set $B \subseteq \omega$ such that $B$ is $\rr$-shift-increasing, $B\neq \emptyset$, and $\forall i \in B \forall j \in \omega \exists n \in B \cap {A}_{i, j}\[n \neq i\]$.
 Applying (1) of Lemma \ref{lem:dense} to $B$ we conclude that $B \in {\III}^{+}$.
 By definition of $\III$, there exists $A \subseteq B$ so that the subspace $\{{y}_{n}: n \in A\}$ is homeomorphic to $\QQ$.
 Let $N = \{{\alpha}_{n}: n \in A\} \subseteq M$.
 As $\PPP$ is hereditary, $N \in \PPP$.
 Clearly, $A$ is an infinite subset of $\omega$, and so $\otp(N) = \omega$.
 By definition $N$ is $\bar{\rho}$-shift-increasing.
 Finally, the subspace $\{{x}_{\alpha}: \alpha \in N\} = \{{x}_{{\alpha}_{n}}: n \in A\} = \{{y}_{n}: n \in A\}$ is homeomorphic to $\QQ$.
 So $N$ is as needed.
\end{proof}
\def\polhk#1{\setbox0=\hbox{#1}{\ooalign{\hidewidth
  \lower1.5ex\hbox{`}\hidewidth\crcr\unhbox0}}}
\providecommand{\bysame}{\leavevmode\hbox to3em{\hrulefill}\thinspace}
\providecommand{\MR}{\relax\ifhmode\unskip\space\fi MR }
\providecommand{\MRhref}[2]{%
  \href{http://www.ams.org/mathscinet-getitem?mr=#1}{#2}
}
\providecommand{\href}[2]{#2}

%\bibliographystyle{amsplain}
%\bibliography{Bibliography} 
\end{document}